\documentclass[11pt]{article}

\usepackage[a4paper]{geometry}
\usepackage{times}
\usepackage[english]{babel}
\usepackage[latin1]{inputenc}
\usepackage{lmodern}
\usepackage[T1]{fontenc}
\usepackage{amsmath,amssymb}
\usepackage{amsthm}
\usepackage{amsfonts}
\usepackage{hyperref}

\usepackage{tikz} 

\DeclareMathOperator{\Em}{EM}
\DeclareMathOperator{\ha}{HA}

\DeclareMathOperator{\lpo}{LLPO}
\DeclareMathOperator{\List}{List}

\DeclareMathOperator{\rt}{RT}

\DeclareMathOperator{\Expl}{Visit}
\DeclareMathOperator{\Exp}{Exp}
\DeclareMathOperator{\Com}{Complete}
\DeclareMathOperator{\Stab}{Stable}
\DeclareMathOperator{\head}{Head}

\newcommand{\llpo}[1]{{#1}\mbox{-}\!\lpo}

\newcommand{\vgt}[1]{``#1''}
\newcommand{\NN}{\mathbb{N}}

\newcommand{\Ap}[1]{\langle #1 \rangle}
\newcommand{\bp}[1]{\left\lbrace #1 \right\rbrace}

\newcommand{\lex}{<_{\mbox{{\tiny lex}}}}

\theoremstyle{plain}
	\newtheorem{theorem}{Theorem}[section]
	\newtheorem{proposition}[theorem]{Proposition}
	\newtheorem{lemma}[theorem]{Lemma}
	\newtheorem{corollary}[theorem]{Corollary}
	\newtheorem*{claim}{Claim}
\theoremstyle{definition}
	\newtheorem{definition}[theorem]{Definition}

	\newtheorem*{RTduek}{$\mathbf{RT^2_k(\Sigma^0_n)}$}
	\newtheorem*{RTunok}{$\mathbf{RT^1_k(\Sigma^0_n)}$}
	\newtheorem*{konig}{$\mathbf{\Sigma^0_n\mbox{-}LLPO}$}
	\newtheorem*{pigeonhole}{Pigeonhole Principle for $\mathbf{\Sigma^0_n}$}
	\newtheorem*{emn}{$\mathbf{EM_n}$}
\theoremstyle{remark}

\DeclareMathOperator{\ramsey}{RT_2^2}
\DeclareMathOperator{\koenig}{\llpo{3}}

\DeclareMathOperator{\demorgan}{DeMorgan(\Sigma^0_n)}
\DeclareMathOperator{\demorgantre}{DeMorgan(\Sigma^0_3)}

\DeclareMathOperator{\RamoInfinito}{InfiniteBranch}
\DeclareMathOperator{\Omogeneo}{HomSet}

\title{Ramsey's Theorem for Pairs and $k$ Colors as a Sub-Classical Principle of Arithmetic}
\author{Stefano Berardi and Silvia Steila}
\date{}

\begin{document}
\maketitle

\begin{abstract}
The purpose is to study the strength of Ramsey's Theorem for pairs restricted to recursive assignments of $k$-many colors, with respect to Intuitionistic Heyting Arithmetic. We prove that for every natural number $k \geq 2$, Ramsey's Theorem for pairs and recursive assignments of $k$ colors is equivalent to the Limited Lesser Principle of Omniscience for $\Sigma^0_3$ formulas over Heyting Arithmetic. Alternatively, the same theorem over intuitionistic arithmetic is equivalent to: for every recursively enumerable infinite $k$-ary tree there is some $i < k$ and some branch with infinitely many children of index $i$.
\end{abstract}

\section{Introduction}

In \cite{Bishop}, Bishop classified as \emph{principle of omniscience} any principle of logic which implies the existence of something that we cannot compute. In particular the Limited Principle of Omniscience (LPO) states that for any binary infinite sequence, either every entry is null or there exists an entry which is not null. A weakening of LPO is Lesser Limited Principle of Omniscience (LLPO) which states that given any binary infinite sequence with at most one non-null entry, either all even entries are null or all odd entries are null. The omniscience lays in the capability to exclude that the null entry, if any exists, is even, or to exclude that it is odd. In \cite{LICS04}, LLPO is reformulated using predicates: if for all $x$, $y$ either $P(x)$ or $Q(y)$, then either $P(x)$ for all $x$, or $Q(y)$ for all $y$. This formulation is equivalent in Heyting Arithmetic ($\ha$) to the principle formulated by Bishop, if we represent functions through their graph. We denote with  $\llpo{(n+1)}$ the principle LLPO restricted to $\Sigma^0_n$ predicates.

Let $k$ be a natural number. \emph{Ramsey's Theorem for pairs in $k$-many colors} \cite{Ramsey} ($\rt^2_k$) states that given any coloring over the edges of the complete graph with countably many nodes in $k$-many colors, there exists an infinite homogenoeus set, i.e. there exist an infinite subset $H$ of the set of nodes and a color $i < k$ such that for any $x,y \in H$ the edge $\bp{x,y}$ has color $i$. Specker proved that there are recursive colorings in $2$-colors with no recursive homogeneous sets \cite{Specker}; hence for every natural number $k \geq 2$, $\rt^2_k$ requires omniscience in the sense of Bishop.

There are several applications of Ramsey's Theorem for pairs through mathematics and computer science. For instance Ramsey's Theorem for pairs has many applications in mathematical analysis \cite{ArgTod} and it is used to prove the complementation of B\"{u}chi's automata \cite{Buchi} and both the Termination Theorem \cite{Podelski} and the Size-Change Termination Theorem \cite{Jones1}, which characterize the termination for some class of programs (see also \cite{Gasarch}). By using constructive consequences of Ramsey's Theorem for pairs instead of Ramsey's Theorem itself, namely Almost-full Theorem \cite{Coquand} and H-closure Theorem \cite{Hclosure}, it is possible to extract bounds for programs proved to be terminating by the Termination Theorem and the Size-Change Termination Theorem \cite{Hclosure, CoquandStop, SCTTypes}. As explained in \cite{LICS04}, the amount of classical logic needed to prove a theorem reflects the hardness of studying its constructive consequences. This is our motivation for establishing how much classical logic is needed to prove Ramsey's Theorem for pairs.

A priori, it is not evident whether a classical principle expressing $\rt^2_k$ in intuitionistic arithmetic exists, but $\rt^2_k$ happens to be related to a fragment of LLPO, for every natural number $k \geq 2$. In \cite{RT22iff3LLPO} we proved that Ramsey's Theorem for pairs with recursive assignments of two colors is equivalent  to the sub-classical principle  $\koenig$ (i.e. LLPO for $\Sigma^0_2$-predicates) over $\ha$. The goal of this paper is to extend this result to any $k\geq2$, proving that $\koenig$ is equivalent to Ramsey's Theorem for pairs with recursive assignments of $k$ colors. On the one hand, since  $\rt^2_k \implies \ramsey$ trivially holds in $\ha$ for every $k \geq 2$, the result of \cite{RT22iff3LLPO} yields that $\rt^2_k$ for recursive colorings implies  $\koenig$. On the other hand proving $\koenig \implies \rt^2_k$ for recursive colorings is a non-trivial task, and the proof for $2$ colors does not generalize.

It is not too difficult to take the proof by Erd\H{o}s and Rado of Ramsey's Theorem for pairs (e.g. \cite{KohlKreuz1, Jockusch}) and to formalize it in $\ha$ extended by some classical principle. In this way we could prove that, for every natural number $k \geq 2$, $\llpo{4}$ (i.e. LLPO for $\Sigma^0_3$-predicates) implies $\rt^2_k$ for recursive colorings, but this is not yet the result we want to establish.
If we look carefully through Erd\H{o}s-Rado's proof, we notice that $\llpo{4}$ is used to prove a combinatorial statement about $k$-ary trees: any infinite recursive tree with edges in $k$ colors, and having all edges from the same node in different colors, has a branch with infinitely many edges in the same color. We propose to call this statement \emph{Ramsey's Theorem for trees}. Hence if we are able to prove Ramsey's Theorem for trees by using only $\ha$ plus $\llpo{3}$, then we may conclude that $\llpo{3}$, Ramsey's Theorem for trees and $\rt^2_k$ are equivalent over $\ha$.

Our solution involves defining a method (new, as far as we know) to explore any $k$-ary tree.  This method provides a subtree from which we can prove Ramsey's Theorem for trees by using only $\llpo{3}$. To this aim a key definition is the $D$-visit of any $k$-ary tree, where $D$ is a list reflecting the priority of colors in the tree.  This part of the paper is self-contained, and proves a combinatorial result on $k$-ary trees in Classical Arithmetic. For any infinite $k$-ary tree $U$ there exists an infinite subtree $T$, recursively enumerable in $U$, such that $T$ has only one infinite branch and this branch reflects any infinite color; namely if there are infinitely many edges in some color $c$ in $T$, then there are infinitely many edges in color $c$ in the infinite branch. We may precise how much classical logic we use: if $U$ is $\Sigma^0_1$, this result may be proved in $\ha + \llpo{3}$. The tree $U$ used in Erd\H{o}s-Rados's proof is $\Sigma^0_1$, therefore using $\llpo{3}$ we may prove there exists an infinite subtree $T$ whose only infinite branch reflects any infinite color of $T$. $T$ is the subtree we use to prove Ramsey's Theorem from $\llpo{3}$ in $\ha$.
 
This is the plan of the paper. In Section \ref{SectionDef} we explain how to state Ramsey's Theorem for pairs without using function and set variables. As in \cite{RT22iff3LLPO}, we drop function and set variables, and we consider  Heyting Arithmetic, in which we have no Excluded Middle Schema but we have the full induction schema. In Section \ref{section: T2k} we introduce $D$-visits and we present how to define the subtree we use, in Section \ref{section:3LLPOimpliesRT2n}, to prove $\koenig \implies \rt^2_k$ for recursive colorings, for every natural number $k$.

\section{Ramsey's Theorem and Classical Principles for Arithmetic}\label{SectionDef}

In this section we introduce some notations for Ramsey's Theorem and for some other classical principles, following the one used in \cite{RT22iff3LLPO}. Any natural number $k$ is identified with the set $\{0,\ \dots,\ k-1\}$. We use $\NN$ to denote the least infinite ordinal, which is identified with the set of natural numbers. For every set $X$ and every natural number $r$,
\[[X]^r =\{Y \subseteq X \mid |Y|=r\}\]
denotes the set of subsets of $X$ of cardinality $r$. If $r=1$ then $[\NN]^r$ is the set of singleton subsets of $\NN$, and just another notation for $\NN$. If $r=2$ then $[\NN]^2$ is the complete graph on $\NN$: we think of any subset $\{x,y\}$ of $\NN$ with $x \neq y$ as an edge of the graph. For us each edge $\{x,y\}$ has direction from $\min\{x,y\}$ to $\max\{x,y\}$. In the general case, the elements of $[\NN]^r$ are called \emph{hyper-edges}, but we are not concerned with hyper-edges in this paper. Let $k, r$ be positive integers, then a map 
$f:[\NN]^r \rightarrow k$
is called a \emph{coloring} of $[\NN]^r$ with $k$ colors. If $r=2$ and $f(\{x,y\})=h < k$, then we say that the edge $\{x,y\}$ has color $h$. If $f:[\NN]^r \rightarrow k$ is a map then for all $X \subseteq \NN$ we denote with $f''[X]^r$ the set of colors of hyper-edges of $X$, that is:
\[f''[X]^r= \{h\in k \mid \exists e \in [X]^r \mbox{ such that } f(e)=h \}.\]
We say that $X \subseteq [\NN]^r$ is \emph{homogeneous} for $f$, or $f$ is homogeneous  on $X$, if $X$ is inhabited and all hyper-edges of $X$ have the same color, that is, there exists $h<k$ such that $f''[X]^r=\{h\}$. We also say that $X$ is homogeneous for $f$ in color $h$.
If $r=1$ we can think of the function $f$ as a point coloring map on natural numbers. In this case a homogeneous set $X$ is any set of points of $\NN$ which all have the same color. If $r=2$ we can think of the function $f$ as an edge coloring of a graph that has as its vertices the natural numbers. In this case a homogeneous set $X$ is any inhabited set of elements of $\NN$ whose connecting edges all have the same color.

We denote Heyting Arithmetic, with one symbol and axioms for each primitive recursive map, with $\ha$. We work in the language for Heyting Arithmetic with all primitive recursive maps, extended with the symbols $\{f_0, \dots, f_n\}$, where $n$ is a natural number and $f_i$ denotes a total recursive function for all $i <  n+1$. These $f_i$ will indicate an arbitrary coloring in the formulation of Ramsey's Theorem below. If $P= \forall x_1 \exists x_2 \dots p(x_1, x_2, \dots)$, with $p$ arithmetic atomic formula, and $Q= \exists x_1 \forall x_2 \dots \neg p(x_1, x_2 \dots)$, then we say that $P$, $Q$ are dual each other and we write $P^\bot=Q$ and $Q^\bot=P$. Dual is defined only for prenex formulas as $P$, $Q$.
We consider the classical principles as statement schemas as in \cite{LICS04}. A \emph{conjunctive schema}
is a set $\mathcal{C}$ of arithmetical formulas, expressing the second order statement ``for all $P$ in $\mathcal{C}$,
$P$ holds'' in a first order language. We prove a conjunctive schema $\mathcal{C}$ in $\ha$ if we prove any $P$
in $\mathcal{C}$ in $\ha$. A conjunctive schema $\mathcal{C}$ implies a formula $P$ in $\ha$ if $s_1 \wedge \dots \wedge s_n \vdash P$ in $\ha$ for some $s_1,\ \dots, \ s_n \in \mathcal{C}$. The conjunctive schema $\mathcal{C}$ implies another conjunctive schema $\mathcal{C}'$ in $\ha$ if $\mathcal{C}$ implies $P$ in $\ha$ for any $P$ in $\mathcal{C}'$. In order to express Ramsey's Theorem we also have to consider the dual concept of \emph{disjunctive schema} $\mathcal{D}$, expressing the second order statement ``for some $P$ in $\mathcal{D}$, $P$ holds'' in a first order language. We prove a disjunctive schema $\mathcal{D}$ in $\ha$ if we prove $s_1 \vee \dots  \vee s_n$ in $\ha$ for some $s_1, \dots , s_n \in \mathcal{D}$. A disjunctive schema $\mathcal{D}$ implies a formula $P$ in $\ha$ if $s \vdash P$ in $\ha$ for all $s \in \mathcal{D}$.

Ramsey's Theorem is a very important result for finite and infinite combinatorics. In this paper we study Ramsey's Theorem in $k$ colors, for singletons and for pairs. They are informally stated as follows:

\begin{RTunok}
For any coloring $c_a:\NN \rightarrow k$ of vertices with a parameter $a$, there exists an infinite subset of $\NN$ homogeneous for the given coloring. ($c_a \in \Sigma^0_n$).
\end{RTunok}
\begin{RTduek}
For any coloring $c_a:[\NN]^2 \rightarrow k$ of edges with a parameter $a$, there exists an infinite subset of $\NN$ homogeneous for the given coloring. ($c_a \in \Sigma^0_n$).
\end{RTduek}

$\ramsey(\Sigma^0_0)$ (respectively $\rt^1_k(\Sigma^0_0)$) says that given a family $\{c_a \mid a \in \NN\}$ of recursive edge (node) colorings  of a graph with $\NN$ nodes, then for any coloring there exists a subgraph with $\NN$ nodes such that each edge (node) of the subgraph has the same color. 

Here $c = \{c_a \mid a \in \NN\}$ denotes any recursive family of recursive assignment of $k$ colors. In this work we formalize Ramsey's Theorem for $k$ colors for pairs (respectively, for singletons) and for recursive colorings by the following disjunctive schema which we call \emph{Ramsey's schema}:
\[\forall a ( \bigvee \bp{ C_i(., c_a) \mbox{ is infinite and homogeneous }\mid i < k}).\]
A sufficient condition to prove Ramsey schema is to find at least $k$-many arithmetical predicates $C_0$, \dots, $C_{k-1}$ and a proof of $\forall a (C_0 (., c_a)$ is infinite homogeneous in color $0$  $\vee \dots \vee \  C_{k-1}(., c_a)$ is infinite homogeneous in color $k-1$) in $\ha$. For short we say that for each recursive family of recursive colorings there is a homogeneous set.

The conjunctive schemata for $\ha$ we consider, expressing classical principles and taken
from  \cite{LICS04}, are the followings.

\begin{konig} Lesser Limited Principle of Omniscience.
For every parameter $a$
\[\ \forall x, x'\ (P(x,a)\ \vee \ Q(x',a)) \implies \forall x  P(x,a)\  \vee\ \forall x Q(x,a). \ (P,\ Q \in \Sigma^0_{n-1})\]
\end{konig}

$\koenig$ is a kind of law for prenex formulas and if we assume the Axiom of Choice it is equivalent to Weak K\"{o}nig's Lemma for $\Sigma^0_{n-1}$ trees.

\begin{pigeonhole}The Pigeonhole Principle states that given a partition of infinitely many natural numbers in two classes, then at least one of these classes has infinitely many elements.
For every parameter $a$
\[ \ \forall x \ \exists z \ [z \geq x \ \wedge \ (P(z,a) \vee Q(z,a))]\implies\]
\[ \forall x \ \exists z \ [z\geq x \ \wedge \ P(z,a)] \ \vee \ \forall x \ \exists z \ [z \geq x \ \wedge \  Q(z,a)]. \ (P, Q \in \Sigma^0_{n})\]
\end{pigeonhole}

\begin{emn} Excluded Middle for $\Sigma^0_n$ formulas.
For every parameter $a$
\[ \exists x \ P(x,a) \ \vee \ \neg \exists x  \ P(x,a).\ (P \in \Pi^0_{n-1}) \]
\end{emn}

Recall that $P^\bot$ denotes the dual of $P$ for any prenex $P$. As shown in \cite[corollary 2.9]{LICS04} the law of Excluded Middle for $\Sigma^0_n$ formulas is equivalent in $\ha$ (that is, only using intuitionistic arithmetical reasoning) to 
\[ \exists x \ P(x,a) \ \vee \ \forall x  \ P(x,a)^\bot .\ (P \in \Pi^0_{n-1})\] 
In all our schemata we use parameters. The parameter $a$ is necessary since we need to use in $\ha$ statements with a free variable $a$, like
$\forall a \ (\forall x \ P(x,a) \ \vee \ \exists x \ \neg P(x,a))$
in our proof.


\section{Complete D-visits} \label{section: T2k}

In this section we define a visit of $k$-ary trees with respect to a list $D$ of priority among colors (an ordering of some subset of all colors), then we prove its properties.

Given an arithmetical predicate $X$, we write $\Em(X)$ to denote the Law of Excluded Middle for predicates of the same complexity as $X$. And $\Em(X+1)$ to denote the Law of Excluded Middle for predicates with one quantifier more than the predicate $X$.  A $k$-tree on colors $C= \Ap{c_0, \dots, c_{k-1}}$ is a subset of the set of finite lists with elements in $C$ which includes the empty sequence and which is closed under prefix. Given two finite lists $\lambda$ and $\mu$ we write $\lambda \prec \mu$ if $\lambda$ is a proper prefix of $\mu$. We use $\lex$ to denote the lexicographical order of finite sequences.

Let $U$ be a $k$-ary tree on $C= \Ap{c_0, \dots, c_{k-1}}$. We define a subtree $T$ of $U$ which satisfies the following properties. Furthermore, the proof requires a limited amount of classical logic: it may be done in $\ha + \Em(U+1)$.
\begin{itemize}
\item[(a)] There exists some predicate in $\ha$ which is $\Sigma^0_1$ with respect to an oracle for $U$, and which represents $x \in T$.
\item[(b)] $T$ has a unique infinite branch $r$ defined by some predicate of $\ha$.
\item[(c)] If there exist infinitely many edges with color $h$ in $T$, then there are infinitely many edges with color $h$ in $r$.
\end{itemize}

The construction is done by induction over $k$. Assume that a lazy artist has a new work to paint. Given a $k$-ary tree with order on colors $C= \Ap{c_0, \dots, c_{k-1}}$, he has to draw a subtree on the wall of a huge room. 
The only wishes of the customer are 
\begin{itemize}
\item he can draw a node only if he has already drawn its father;
\item if $k>0$, hence the tree has to be infinite.
\end{itemize}
Of course, if $k=0$ the artist can just draw the root $x_0$ and then he is done.
Assume that the artist already knows how to paint a subtree for a $k$-ary tree and he had to paint a subtree of a $(k+1)$-tree. He decides to avoid using colors with lower index, like $c_0$, whenever it is possible. Since he is very lazy,  he decides to avoid changing the color he is currently using, whenever it is possible. So when the artist starts using the color $c_0$, which he initially avoided, he tries no changing it any more, and put $c_0$ at the end of his color list, with the highest index and therefore with top priority. As a result, he paints a very peculiar subtree, which we call complete $C$-visit of $U$.

\begin{enumerate}
\item Given root $x_0$ and an order on colors $\Ap{c_0, \dots, c_{k}}$, first he paints a tree for colors $\Ap{c_1, \dots, c_{k}}$ with root $x_0$. If it is infinite he is done. Otherwise let $L$ be an enumeration of the nodes he obtained.
\item Given $i < |L|$, he chooses a $c_0$-child $y$ of the $i$-th node in $L$ (if it exists) and he performs step (1) with root $y$ and colors $\Ap{c_1, \dots, c_{k},c_0}$. If it is infinite he is done. Otherwise he does the same for $i+1$.
\end{enumerate}

Assume that $k=1$ (the only color is $0$). Then the tree is a straight line and the artist paints all its edges in a row. Assume that $k=2$ (the only colors are $0$ and $1$). By unfolding the previous rules, the artist first paints the longest line avoiding all edges in color $1$. If this branch is finite, he paints the first line with all edges in the color $0$ departing from some node in it, and so forth. Whenever he starts painting some subtree, he moves out of it only if he completely paints it, otherwise he continues forever inside the same subtree. When $k=3$ (the only colors $0,1,2$), the artist starts with the colors $1,2$ and the color $0$ shows up only if it is impossible to paint forever from the root using $1,2$ only. In this case the order between the colors is permuted cyclically, now the painter uses $0,2$ in preference, and $1$ only when is forced to. 

Although it is not self-evident, we claim that the tree painted by the artist satisfies the Properties (a), (b), (c); namely it is $\Sigma^0_1$ in $U$, it has a unique infinite branch $r$ and if there exist infinitely many edges with color $h$ in the tree, then there are infinitely many edges with color $h$ in $r$.

\begin{theorem}\label{theorem:k+1}
Let $(U,\prec)$  be a $k$-ary tree. Then there exists $f_C: \NN \to \NN$ recursive enumerable in $U$ and such that in $\ha + \Em(U+1)$ we can prove the following.
\begin{enumerate}
\item $(f_C''\NN, \prec)$ is a subtree of $U$;
\item if $U$ is infinite, then $f_C''\NN$ is infinite;
\item $f_C''\NN$ satisfies Properties {\normalfont (a), (b), (c)}.
\end{enumerate}
\end{theorem}

\subsection{Proof of Theorem \ref{theorem:k+1}}
Throughout this section assume given an arithmetical $k$-ary tree $U$ with colors in $C=\Ap{c_0, \dots, c_{k-1}}$. We write $D \subseteq C$ when $D$ is a list composed of distinct elements of $C$, possibly in an order different from the order they have in $C$.

\begin{definition}
Given a finite list $L$ of nodes in $U$ and a color $c$, we define the predicate $\Com(L,c)$ (and we say that \emph{$L$ is $c$-complete}) if any child of color $c$ in $U$ of a node in $L$ is also in $L$:
\[
 	\forall \lambda \in L ( \lambda*\Ap{c} \in U \implies \lambda*\Ap{c} \in L ).
\]
Given $D \subseteq C$, we say that \emph{$L$ is $D$-complete} if $L$ is $d$-complete, for all $d \in D$.
\end{definition}

\begin{definition}
Given a finite list $L$ of nodes in $U$, $\mu \in L$, natural number $n$ and a color $c$, we define the predicate \emph{$\mu*\Ap{c}$ is the $n$-th $c$-expansion of $L$}, which we write $\Exp(\mu*\Ap{c}, L,n,c)$. The definition is by induction on $n$.
$\Exp(\mu*\Ap{c},L,n,c)$ holds if $\mu*\Ap{c}$ is the $n$-th node of $U$, in the lexicographic ordering, having the form $\nu*\Ap{c}$ for some $\nu$:
\begin{itemize}
\item  $n=0$ and $\mu*\Ap{c} \in U$ and $\forall \eta \lex \mu \neg (\eta*\Ap{c} \notin U)$. 
\item  $n>0$ and $\mu*\Ap{c} \in U$ and $\exists \eta_0, \dots, \eta_{n-1} \lex \mu $ such that $(\forall i \in n (\eta_i*\Ap{c} \in U)$ and $\forall \eta < \mu (\eta*\Ap{c} \in U \implies \exists i \in n (\eta =\eta_i))$.  
\end{itemize}
\end{definition}

In order to define a tree $T$ as required by Theorem \ref{theorem:k+1}, instead of the lexicographic ordering $\lex$ on lists of $C$ we can use any total ordering of $U$. We introduce now a notion of visit of $U$ we use to enumerate the nodes of $T$ (it is the drawing rule selected by the \vgt{artist}). In order to provide a recursive definition, for any subset $D=\Ap{d_0, \dots, d_{h-1}}$ of the set of colors $C$ and for any $\lambda \in U$, we define visits of $U$ with priority $D$ of colors from $\lambda$ (just $D$-visits from $\lambda$ for short). Informally a $D$-visit of $U$ defines a finite subtree of $U$ in such a way that
\begin{itemize}
\item the color $d_0$ has the lowest priority, we try to use only $\Ap{d_1, \dots, d_{n-1}}$;
\item when forced to use $d_0$, we give to it the highest priority; i.e. we keep working with priority $\Ap{d_1, \dots, d_{n-1}, d_0}$.
\end{itemize}

\begin{definition} Given $D = \Ap{d_0, \dots, d_{h-1}} \subseteq C$ we define the predicate $\Expl(L,D,\lambda)$ on finite lists (and we say \emph{$L$ is a $D$-visit from $\lambda$}) by principal induction over $h$ and secondary induction over the size of $L$. $\Expl(L,D,\lambda)$ holds if either
\begin{itemize}
\item $h=0$ and $L= \Ap{\lambda}$.
\item $h>0$ and there exists $n < |L|$ such that $L  = M* L_0 * \dots * L_{n-1}$ and the following hold:
\begin{itemize}
\item $M$ is a $\Ap{d_1, \dots, d_{h-1}}$-visit from $\lambda$, i.e. $\Expl(M,\Ap{d_1,\dots, d_{h-1}}, \lambda)$.
\item If $n>1$, $M$ is $\Ap{d_1, \dots, d_{h-1}}$-complete.
\item For all $j < n$, $\head(L_j)$ is the $j$-th expansion of $M$ in color $d_0$ and $L_j$ is a $\Ap{d_1, \dots, d_{h-1},d_0}$-visit from $\head(L_j)$, namely 
\[
    (\Exp(\head(L_j),M, j, d_0) \wedge \Expl(L_j,\Ap{d_1,\dots, d_{h-1},d_0},\head(L_j))).
\]
\item For all $j < n-1$, $L_j$ is $D$-complete.
\end{itemize}
\end{itemize}
\end{definition}

It is straightforward to show directly that $\Expl(L,D,\lambda)$ is recursive in $U$, and that any element of $L$ but the first is the descendant number $c$ of some previous node of $L$, for some $c \in D$. We may prove that any $D$-visit $L$ of $U$ from a node $\lambda$ is some one-to-one enumeration of some subtree of $U$ of root $\lambda$, with all edges with colors in $D$. 

\begin{lemma}\label{lemma:subtree}
Let $D = \Ap{d_0, \dots, d_{h-1}} \subseteq C$ and let $\lambda \in U$ and let $L$ be a $D$-visit from $\lambda$. Then $L$ has no repetitions and $(\bp{\mu \mid \mu \in L}, \prec)$ is a subtree of $U(\lambda)$.
\end{lemma}
\begin{proof}
By unfolding definition we have to show that for every $L$:
\[
	\forall \mu \in L \forall \eta \prec \mu( \lambda \preceq \eta \implies \eta \in L ).
\]
We prove it by induction on $h$.  If $h=0$, then $L= \Ap{\lambda}$. The thesis follows. 

If $h>0$, we prove our goal by secondary induction on $|L|$. By definition $L=M*L_0*\dots*L_{n-1}$. If $\mu \in M$, by inductive hypothesis on $h-1$, we have $ \forall \eta \prec \mu( \lambda \prec \eta \implies \eta \in M \subseteq L)$. Assume that $\mu \in L_j$ for some $j < n$. Since $\lambda$ is not in $L_j$, by secondary inductive hypothesis $L_j$ is a subtree of $U(\head(L_j))$. Hence we have $\forall \eta \prec \mu( \head(L_j) \prec \eta \implies \eta \in L_j \subseteq L)$. Since by construction for all $j < n$ there exists $\xi$ in $M$ such that $\head(L_j)=\xi*\Ap{d_0}$ we are done.

We prove that $L$ has no repetitions by induction on $h$. If $h=0$, then $L=\Ap{\lambda}$ and we are done. If $h>0$, we proceed by secondary induction on $|L|$. By definition $L=M*L_0*\dots*L_{n-1}$, with $M, L_0, \dots L_{n-1}$ pairwise disjoint. Hence any repetition in $L$ is a repetition either in $M$ or in some $L_j$. The thesis follows since we cannot have repetitions neither in $M$, by induction hypothesis on $h-1$, nor in $L_j$, by secondary induction hypothesis.
\end{proof}

Given two finite lists $\eta$ and $\mu$ we write $\eta \vartriangleleft \mu$ to mean that $\eta$ is a suffix of $\mu$. 

\begin{definition}
Let $D = \Ap{d_0, \dots, d_{h-1}} \subseteq C$ and let $\lambda \in U$. We define the \emph{$D$-subtree of $U(\lambda)$} as the subtree of all branches of $U$ with colors in $D$:
\[
	\mu \in U_D(\lambda) \iff  \exists \eta \vartriangleleft \mu ( \mu = \lambda * \eta  \ \wedge \  \eta \in \List( \Ap{d_0, \dots, d_{h-1}})
\]
\end{definition}

Notice that for every $\lambda \in U$, $U_D(\lambda)$ is recursive in $U$. Any $D$-complete $D$-visit of $U$ from some node $\lambda$ happens to cover $U_D(\lambda)$. 

\begin{lemma}\label{lemma:subtreeD}
Let $D = \Ap{d_0, \dots, d_{h-1}} \subseteq C$ and let $\lambda \in U$ and let $L$ be a $D$-visit from $\lambda$. $(\bp{\mu \mid \mu \in L},\prec)$ is a subtree of $U_D(\lambda)$, and it is equal to $U_D(\lambda)$ if $L$ is $D$-complete.  
\end{lemma}
\begin{proof}
Let $L$ be a  $D$-visit from $\lambda$. By Lemma \ref{lemma:subtree}, $((\bp{\mu \mid \mu \in L},\prec), \prec)$ is a subtree of $U(\lambda)$. By definition of $D$-visit any edge in $((\bp{\mu \mid \mu \in L},\prec), \prec)$ has color in $D$, thus we obtain that $L$ is a subtree of $U_D(\lambda)$. 

Moreover assume that $L$ is $D$-complete. Then, by definition, every child in $U_D(\lambda)$ of a node in $L$ is in $L$. It is straightforward to show by induction on $m$ that
\[
	\forall m  \forall \mu \in U_D(\lambda) ( |\mu|= |\lambda| + m  \implies \mu \in L).
\]
Indeed if $m=0$, then $\mu = \lambda$ and $\lambda \in L$. Otherwise if $m>0$, then $\mu = \xi * \Ap{d_i} \in U_D(\lambda)$, for some $i \in h$ and $\xi \in U_D(\lambda)$. By induction hypothesis, $|\xi| = m-1$, $\xi \in L$. By $d_i$-completeness, $\xi *\Ap{d_i} \in L$. 
\end{proof}

As a corollary any $D$-complete $D$-visit from some node $\lambda$ is maximal over the $D$-visits from $\lambda$ with respect to the prefix order of lists.

\begin{corollary}\label{lemma:complprefix}
Let $D = \Ap{d_0, \dots, d_{h-1}} \subseteq C$, let $\lambda \in U$ and let $L$, $L'$ be two $D$-visits from $\lambda$. If $L$ is a prefix of $L'$ and $L$ is $D$-complete then $L=L'$
\end{corollary}
\begin{proof}
By Lemma \ref{lemma:subtreeD}, $(\bp{\mu \mid \mu \in L},\prec)$ and $(\bp{\mu \mid \mu \in L'},\prec)$ are subtrees of $U_D(\lambda)$. Moreover, since $L$ is $D$-complete, by Lemma \ref{lemma:subtreeD} every $\mu \in U_D(\lambda)$ belongs to $L$.  Assume by contradiction that $L = L * \Ap{\mu} \leq L'$ for some $\mu$. Since $\mu \in U_D(\lambda)$, we have $\mu \in L$. Hence in $L'$ there is a repetition, but this contradicts Lemma \ref{lemma:subtree}.
\end{proof}

Applying Corollary \ref{lemma:complprefix}, we may show that the prefix is a linear order over the $D$-visits from $\lambda$.

\begin{proposition}\label{proposition:prefix}
	Let $D = \Ap{d_0, \dots, d_{h-1}} \subseteq C$, let $\lambda \in U$ and let $L$, $L'$ be two $D$-visits from $\lambda$. Then $L$ and $L'$ are comparable by prefix. 
\end{proposition}
\begin{proof}
By induction over $h$.  If $h=0$ then $L = L'=\Ap{\lambda}$, hence the thesis. 

If $h>0$, since $L$, $L'$ are $D$-visits from $\lambda$, we have $L=M*L_0*\dots*L_{m-1}$ and $L'=M'*L_0'*\dots*L_{n-1}'$ as in the definition. We may assume without loss of generality that $m \leq n$. 

If $m=n=0$, then $L= M $ and $L'= M'$. Hence, by induction hypothesis on $h-1$, $M$ and $M'$ are comparable by prefix. 

If $0= m < n$, then $M'$ is $\Ap{d_1,\dots, d_{h-1}}$-complete. By Corollary \ref{lemma:complprefix}, $L = M \leq  M' \leq L'$. 

If $0 < m \leq n$, we prove the thesis by secondary induction on $\max\bp{|L|, |L'|}$. By Corollary \ref{lemma:complprefix}, $M = M'$ since they are $\Ap{d_1,\dots, d_{h-1}}$-complete and comparable by prefix. Since the $d_0$-extensions of $M$ are the same, $\head(L_i)= \head(L_i')$  for all $i < m$.  Since $\lambda$ does not belong to $L_i$ and $L_i'$, we have $|L_i|<|L|$ and $L_i'|<|L'|$. Therefore $\max\bp{|L_i|, |L_i'|} < \max\bp{|L|, |L'|}$, and by secondary induction hypothesis we have $L_i$ and $L_i'$ are comparable by prefix for all $i < m$.
\begin{itemize}
\item If $i< m-1$ then $L_i$ and $L_i'$ are $D$-complete, therefore by Corollary \ref{lemma:complprefix}, $L_i=L_i'$.
\item Assume now that $i= m = n$. Then $L$ and $L'$ have the same prefix $M*L_0*\dots*L_{m-2} = M'* L_0'*\dots*L_{m-2}'$ and $L_{m-1}$, $L_{m-1}'$ are comparable by prefix. Thus $L$ and $L'$ are comparable by prefix. 
\item If $i= m < n$, $L'_{m-1}$ is $D$-complete. Therefore $L_{m-1} \leq L_{m-1}'$ by Corollary \ref{lemma:complprefix}. It follows: 
\[
	L= M*L_0*\dots*L_{m-2}* L_{m-1} \leq  M'* L_0'*\dots*L_{m-2}' * L_{m-1}' \leq L'. \qedhere
\]
\end{itemize}
\end{proof}

Given a $D$-visit $L$ which is not $D$-complete, we can extend it to another $D$-visit.

\begin{lemma}\label{lemma:infinite}
Let $D = \Ap{d_0, \dots, d_{h-1}} \subseteq C$, let $\lambda \in U$ and  let $L$ be a $D$-visit from $\lambda$.  Either $U_D(\lambda) = \bp{\eta \mid \eta \in L}$ and $L$ is $D$-complete, or $U_D(\lambda) \supset \bp{\eta \mid \eta \in L}$ and there exists $\mu \in U$ such that $L*\Ap{\mu}$ is a $D$-visit from $\lambda$.
\end{lemma}
\begin{proof}
By induction over $h$. If $L$ is $D$-complete, by Lemma \ref{lemma:subtreeD} we deduce $U_D(\lambda) =  \bp{\eta \mid \eta \in L}$. If $h =0$, then $L = \Ap{\lambda}$ is $D$-complete and $U_D(\lambda)=\bp{\lambda}$. 

Assume that $L$ is not $D$-complete and $h>0$. We argue by induction on $|L|$. By definition we have $L=M*L_0*\dots*L_{n-1}$. If $n=0$, by Lemma \ref{lemma:subtreeD}, we have $U_{\Ap{d_1, \dots, d_{h-1}}}(\lambda) \supseteq  \bp{\eta \mid \eta \in L}=\bp{\eta \mid \eta \in M }$.  By $\Em(U)$ we may decide whether
\[
	\exists i \in [1,d-1] \neg \Com(M,d_i) \vee \forall i \in [1,d-1] \Com(M,d_i)
\]
In the first case we have $U_{\Ap{d_1, \dots, d_{h-1}}}(\lambda) \supset  \bp{\eta \mid \eta \in L}=\bp{\eta \mid \eta \in M }$, hence can apply the induction hypothesis on $h-1$ for $M$ and we get the thesis. In the second case, $M$ is $\Ap{d_1, \dots, d_{h-1}}$-complete. Since $U_{\Ap{d_1, \dots, d_{h-1},d_0}}(\lambda) \supset \bp{\eta \mid \eta \in L} = \bp{\eta \mid \eta \in M}$ there exists at least one $d_0$-expansion of $L_0$. Choose the first $\mu$ such that $\Exp(\mu, M ,d_0)$.
Now assume that $n>0$. By definition of $D$-visit, for every $j < n-1$ we have $L_j$ is $D$-complete. By secondary induction hypothesis on $L_j$, for all $j <n-1$  we have $U_{\Ap{d_0, \dots, d_{h-1}}}({\head{L_{j}}})= \bp{\eta \mid \eta \in L_{j}}$. By $L$ not $D$-complete we have $L_{n-1}$ not $D$-complete, therefore $U_{\Ap{d_1, \dots, d_{h-1}, d_0}}({\head(L_{n-1})}) \supset\bp{\eta \mid \eta \in L_{n-1}}$. By secondary induction hypothesis on $L_{n-1}$ we get the thesis. 
\end{proof}

As a consequence of both Proposition \ref{proposition:prefix} and Lemma \ref{lemma:infinite}, if $U_D(\lambda)$ is infinite we may prove that any $D$-visit from $\lambda$ can be uniquely extended to another $D$-visit from $\lambda$.

\begin{theorem}\label{theorem:existenceuniqueness}
Let $D = \Ap{d_0, \dots, d_{h-1}} \subseteq C$ and $\lambda \in U$ be such that $U_D(\lambda)$ is infinite. For every $D$-visit $L$ from $\lambda$, there exists a unique $\mu$ such that $L *\Ap{\mu}$ is a $D$-visit from $\lambda$.
\end{theorem}
\begin{proof}
Existence. Lemma \ref{lemma:infinite} and $U_D(\lambda)$ infinite yield the existence of $\mu$ for every given $L$ finite.  

Uniqueness. If  $L *\Ap{\mu_1}$ and $L*\Ap{\mu_2}$ are $D$-visits from $\lambda$, by Proposition \ref{proposition:prefix} they are comparable by prefix. Hence $\mu_1=\mu_2$.
\end{proof}

Theorem \ref{theorem:existenceuniqueness} guarantees that the following function, which we call \emph{complete $D$-visit of $U$} with abuse of notation, is well-defined:
\begin{definition}
Given $D = \Ap{d_0, \dots, d_{h-1}} \subseteq C$ and $\lambda \in U$ be such that $U_D(\lambda)$ is infinite, define $f_D^\lambda: \NN \to \List(C)$, the enumeration of the complete $D$-visit of $U$ from $\lambda$, as follows:
\begin{multline*}
	f_D^\lambda(m)= \mu \iff (m=0 \ \wedge\ \mu=\lambda) \ \vee \ (m>0 \ \wedge \\
	\exists \mu_0,\dots, \mu_{m-1}(\bigwedge_{i=0}^{m-1}(f_D^\lambda(i)=\mu_i) \wedge \Expl(\Ap{\mu_0, \dots, \mu_{m-1},\mu},D,\lambda))).
\end{multline*}
\end{definition}

Note that $f_D^\lambda$ is recursive in $U$. Indeed it may be defined by minimalization: if $U$ is infinite, by Theorem \ref{theorem:existenceuniqueness} there is always some unique new element to be added to the sequence. Thus $f_D^\lambda$ defines a recursive enumerable subtree of $U_D(\lambda)$ by Lemma \ref{lemma:subtreeD}. We suppose that there is an efficient algorithm enumerating the tree $f^\lambda_D(\NN)$, using stacks and pointers, but to design it is out of the scope of this paper. Notice that $f_C^{\Ap{}}$ is the strategy of the \vgt{artist}.

In order to prove that $T=f^{\Ap{}}_C{}'' \NN$ has a unique infinite branch which satisfies the Property (c), we define the following predicate.

\begin{definition}
Assume that $D = \Ap{d_0, \dots, d_{h-1}} \subseteq C$ and $\lambda \in U$ be such that $U_D(\lambda)$ is infinite. A node $\mu$ is \emph{stable} for $\lambda$ if all nodes after $\mu$ in $f_D^\lambda{}''\NN$ are descendants of $\mu$. 
\[
	\Stab(\mu) \iff \exists m \forall n > m \forall \eta ((f_D^\lambda(m) = \mu \wedge f_D^\lambda(n) = \eta) \implies \mu \prec \eta).
\]
\end{definition}

Our next goal is to show that the set of ancestors of stable nodes forms the unique infinite branch of $T$, when $T$ is infinite. The first step is to prove that stable nodes form a straight line in any infinite $D$-visit.

\begin{lemma}\label{lemma:stablebasic}
Assume that $D = \Ap{d_0, \dots, d_{h-1}} \subseteq C$ and $\lambda \in U$ be such that $U_D(\lambda)$ is infinite. 
\begin{enumerate}
\item If $\mu$ is stable, then $U_D(\mu)$ is infinite.
\item If $\mu_1$ and $\mu_2$ are stable, then they are comparable by prefix.
\end{enumerate}
\end{lemma}
\begin{proof}
1. By definition of $\Stab(\mu)$, there exists $m$ such that $f_D^\lambda(m)=\mu$ and for every $n>m$, $\mu \prec f_D^\lambda(n)$. Therefore $f_D^\lambda{}''\bp{n \mid n>m} \subseteq U_D(\mu)$. Since all elements in $f_D^\lambda{}''\bp{n \mid n>m}$ are pairwise distinct by definition and Lemma \ref{lemma:subtree} we are done.
 
2. Let $m_1$ and $m_2$ be such that $f_D^\lambda(m_1)=\mu_1$ and  $f_D^\lambda(m_2)=\mu_2$. Without loss of generality we may assume that $m_1 \leq m_2$. If $m_1 < m_2$, $\Stab(\mu_1)$ yields $\mu_1 \prec \mu_2$.
\end{proof}

We can prove that there are infinitely many stable nodes in any infinite $D$-visit from $\lambda$. To this aim we first prove that if a visit contains some $d_0$-child, it contains some stable node which is a $d_0$-child.

\begin{lemma}\label{lemma:stable0}
Let $D = \Ap{d_0, \dots, d_{h-1}} \subseteq C$ and let $\lambda \in U$ such that $U_D(\lambda)$ is infinite. Assume that the visit contains two nodes, one the $d_0$-child of the other, then the visit contains some node $f^\lambda_D(n)$, with $n > 0$, which is stable and $d_0$-child of some other node of the visit.
\end{lemma}
\begin{proof}
Define $L= \Ap{f_D^\lambda(0), \dots, f_D^\lambda(p)}$. Assume that the visit $L$ contains an edge in color $d_0$, then $L = M* L_0 * \dots *L_{m-1}$, for $m>0$. Hence, by definition, $M$ is $\Ap{d_1,\dots, d_{h-1}}$-complete. 
Assume that $|M| = l$. Let $t$ be the maximum such that the expansion number $t$ of $M$ exists, that is, such that $\Exp(\mu,t,M)$ for some $\mu$. We can prove such maximum exists by the following statements of $\Em(U+1)$ for $i < l$:
\[
	\exists \mu (\Exp(\mu, i ,M)) \vee \forall \mu (\neg \Exp(\mu, i ,M)).
\]
For all $i \leq t \leq l$, let $\mu_i$ be the expansion number $t$ of $M$, that is, the unique witness of $\Exp(\mu_i,i,M)$. Thus, by at most $l$-many instances of the following statement of $\Em(U+1)$
\[
	\exists n (f_D^\lambda(n)=\mu_i) \vee  \forall n (f_D^\lambda(n) \neq \mu_i)
\]
We find the greater $n$ such that $f_D^\lambda(n)$ is a $d_0$-expansion of $M$. Such node is stable by unfolding definitions. Hence $n$ and $f_D^\lambda(n)$ are the wished witnesses.
\end{proof}

By Lemma \ref{lemma:stable0}, as long as we find $d_0$-children we find stable nodes which are $d_0$-children. From this remark and induction over the number of colors we may prove:

\begin{proposition}\label{corollary:stable}
Let $D = \Ap{d_0, \dots, d_{h-1}} \subseteq C$ and let $\lambda \in U$ such that $U_D(\lambda)$ is infinite. If $f_D^\lambda(m)$ is stable for $\lambda$, there exists $n >m$ such that $f_D^\lambda(n)$ is stable for $\lambda$.
\end{proposition}
\begin{proof}
By induction on $h$. If $h=0$ is trivial since the only node of $U_D(\lambda)$ is $\lambda$. 
Let $h > 0$. By $\Em(U+1)$ either there are two nodes of index greater or equal to $m$ in $f^\lambda_D(\NN)$, one the $d_0$-child of the other, or not:
\begin{align*}
	&\forall m_0 \geq m \forall m_1 > m_0 \forall \mu(f_D^\lambda(m_0)=\mu \wedge f_D^\lambda(m_1)\neq \mu * \Ap{d_0}) \vee \\
	&\exists m_0 \geq m \exists m_1 >m_0 \exists \mu (f_D^\lambda(m_0)=\mu \wedge f_D^\lambda(m_1)=\mu * \Ap{d_0}).
\end{align*}
The first case yields $U_{\Ap{d_1,\dots, d_{h-1}}}(f_D^\lambda(m))$ is infinite. Hence by inductive hypothesis on $h-1$, we get there exists $n > m$ such that $f_D^\lambda(n)$ is stable for $f_D^\lambda(m)$. Hence $f_D^\lambda(n)$ is stable also for $\lambda$.

Otherwise, given the witnesses of 
\[
	  \exists m_0 \geq m \exists m_1 \exists \mu (f_D^\lambda(m_0)=\mu \wedge f_D^\lambda(m_1)=\mu*\Ap{d_0}),
\]
we obtain some edge in color $d_0$ in $L$. Hence, by applying Lemma \ref{lemma:stable0} to $U_D(f_D^\lambda(m))$ which is infinite since $f_D^\lambda(m)$ is stable (Lemma \ref{lemma:stablebasic}.1), we have that $f_D^\lambda(n)$ is the wished witness, because $n > m$.
\end{proof}

From Lemma \ref{lemma:stable0} and Proposition \ref{corollary:stable} we will prove that any color occurring infinitely many times in a $D$-visit occurs infinitely many times between some node and some stable node. This is to say: given a color $d_i$, if there are infinitely many $d_i$-children in the $D$-visit from $\lambda$, then there are infinitely many $d_i$-children which are stable.  The proof is direct only for the color $d_0$, for a generic color $d_i$ we use induction over the color position in the list $D$.

\begin{proposition}\label{proposition:stable2}
Let $D = \Ap{d_0, \dots, d_{h-1}} \subseteq C$ and let $\lambda \in U$ such that $U_D(\lambda)$ is infinite. Let $i <h$. Assume that there are infinitely many nodes in the $D$-visit from $\lambda$ which are $d_i$-child of some other node of the visit. Then there are infinitely many stable nodes in the $D$-visit from $\lambda$ which are $d_i$-child of some other node of the visit.
\end{proposition}
\begin{proof}

We prove, by induction over $i$, that for every $q$ there is some stable node $f_D^\lambda(m)$, with $m> q$, in the $D$-visit from $\lambda$ which is a $d_i$-child of some other node of the visit. 

Let $i=0$ and fix a natural number $q$ in order to prove that there exists $m> q$ such that $f_D^\lambda(m)$ is both stable and a $d_0$-child of some other node of the visit.  By Proposition \ref{corollary:stable} there is $p > q$ such that $f_D^\lambda(p)$ is stable. Hence $U_D(f_D^\lambda(p))$ is infinite by Lemma \ref{lemma:stablebasic}.1. Since there are  $p < m_0 < m_1$ such that $f_D^\lambda(m_1)$ is a $d_i$-child of $f_D^\lambda(m_0)$, by Lemma \ref{lemma:stable0} there exists $m > p > q$ such that $f_D^\lambda(m)$ is both stable and a $d_0$-child for some other node of the visit.

Assume that the thesis holds for $i$.  By secondary induction over $h$.  If $h=0$, then $U_D(\lambda)$ has a unique node $\lambda$, hence the thesis. 

Assume that $h>0$ and fix a natural number $q$. By Proposition \ref{corollary:stable} there is $p > q$ such that $f_D^\lambda(p)$ is stable. Hence $U_D(f_D^\lambda(p))$ is infinite by Lemma \ref{lemma:stablebasic}.1. Let  $p< m_0< m_1$ be such that $f_D^\lambda(m_1)$ is a $d_i$-child of $f_D^\lambda(m_0)$. If there are no $n > p$ such that $f_D^\lambda(n)$ is a $d_0$-child of some other node, then for every $n' > p$, $f_D^\lambda(n')= f_{\Ap{d_1,\dots, d_{h-1}}}^\lambda(n')$. Then $U_{\Ap{d_1,\dots, d_{h-1}}}(f_D^\lambda(p))$ is infinite and by induction hypothesis on $h-1$ we are done.

Otherwise, by Lemma \ref{lemma:stable0}, there exists some $n > p$ such that $f_D^\lambda(n)=\mu$ is both stable and a $d_0$-child of some other node of the visit. Hence by definition, for every $j$, $f_D^\lambda(n+j) = f_{\Ap{d_1,\dots,d_{h-1},d_0}}^\mu(j)$ and  $U_{\Ap{d_1,\dots,d_{h-1},d_0}}(\mu)$ is infinite. By hypothesis for every $n \leq n + q'$ there exist some $n+ q' < n+ q' + m_0 < n+ q' +m_1$ such that $f_{\Ap{d_1,\dots,d_{h-1},d_0}}^\mu(q' + m_1) = f_D^\lambda(n +q'+m_1) = f_D^\lambda(n +q' + m_0)*\Ap{d_i} = f_{\Ap{d_1,\dots,d_{h-1},d_0}}^\mu(q'+ m_0)*\Ap{d_i}$. The position of $d_i$ in the list $\Ap{d_1, \dots, d_{h-1},d_0}$ is $i-1$. Thus by inductive hypothesis on $i$ we get the thesis. 
\end{proof}

We can now show that $f_D^{\Ap{}}{}''\NN$ has a unique branch which reflects any infinite color. The entire construction requires only the sub-classical principle $\Em(U+1)$, Excluded Middle over predicates with one quantifier more than in the definition on $U$. From this remark we will show that we may prove Ramsey's Theorem for recursive colorings using only $\koenig$.

\begin{theorem}[$\ha + \Em(U+1)$]
Let $D = \Ap{d_0, \dots, d_{h-1}} \subseteq C$ and let $\lambda \in U$ such that $U_D(\lambda)$ is infinite. $f_D^\lambda{}''\NN$ has a unique infinite branch $r$ such that it satisfies Property {\normalfont (c)}.
\end{theorem}
\begin{proof}
Existence. Define $r$ as the closure of the set of all stable nodes. Namely
\[
	\mu \in r \iff \exists \eta (\mu \preceq \eta \wedge \Stab(\eta))
\]
By Lemma \ref{lemma:stablebasic}.2, $r$ is linearly ordered. By Proposition \ref{corollary:stable} we have that $r$ is infinite.

Uniqueness. 
Assume that $r_1$ and $r_2$ are infinite branches and, for every $k \in \NN$, denote by $r_{i}(k)$ the node of height $k$ in $r_i$. We claim that for every natural number $k$, $r_1(k)=r_2(k)$. Fix $k \in \NN$ and $i \in \bp{1,2}$. By Proposition $\ref{corollary:stable}$ there exists $n$ greater than the indexes of $r_1(k)$ and $r_2(k)$ such that $f_D^\lambda(n)$ is a stable. Since $r_i(k)$ has infinitely many proper descendants (which are pairwise distinct by Lemma \ref{lemma:subtree}), there exists $n_i > n$ such that $r_i(k) \prec f_D^\lambda(n_i)$. Moreover by $\Stab(f_D^\lambda(n))$ we have $f_D^\lambda(n) \prec f_D^\lambda(n_i)$. Hence $r_i(k)$ and $f_D^\lambda(n)$ are comparable by prefix and distinct. By Lemma \ref{lemma:subtree} there are no repetition in $f_D^\lambda(\NN)$, therefore $r_i(k) \prec f_D^\lambda(n)$.  Thus $r_1(k)$ and $r_2(k)$ are ancestors of $f_D^\lambda(n)$ of the same height. Hence $r_1(k) = r_2(k)$.

$r$ satisfies Property (c). By Proposition \ref{proposition:stable2}, if there are infinitely many edges in color $d_i$, then we can find infinitely many stable nodes which are $d_i$-children. Hence we have infinitely many nodes in $d_i$ which are $d_i$-children.
\end{proof}

Note that the branch $r$ defined in the proof above is $\Delta^0_2(U)$. In fact $r$ contains all the nodes with infinitely many descendants. 
\[
	\mu \in r \iff \forall m \exists n > m \exists \eta (f_D^\lambda(n)=\eta \wedge \mu \prec \eta ).
\]


\section{From omniscience to homogeneous sets}\label{section:3LLPOimpliesRT2n}

Let $k \geq 2$ be a fixed natural number. We modify Erd\H{o}s-Rado's proof of $\rt^2_k$ (see e.g. \cite{KohlKreuz1}) to obtain a proof of $\koenig \implies  \rt^2_k(\Sigma^0_0)$ over $\ha$. It is enough to prove that if $\{c_a \mid a \in \NN\}$ is a recursive family of recursive colorings, a finite number of statements in $\koenig$ imply that there are predicates $C_0(.,c), \dots, C_{k-1}(.,c)$ such that, 
\[\forall a ( \bigvee \bp{ C_i(., c_a) \mbox{ is infinite and homogeneous }\mid i < k}).\]

We first sketch Erd\H{o}s-Rado's proof of $\rt^2_k$. It consists in defining a suitable infinite $k$-ary tree $V$. We first remark that $\rt^1_k$ (Ramsey's Theorem for colors and points of $\NN$) is nothing but the Pigeonhole Principle: indeed, if we have a partition of $\NN$ into $k$-many colors, then one of these classes is infinite. We informally prove now $\rt^2_k$ from $\rt^1_k$. Fix any coloring $f:[\NN]^2 \rightarrow k$ of all edges of the complete graph having support $\NN$. If $X$ is any subset of $\NN$, we say that $X$ defines a $1$-coloring of $X$ if for all $x \in X$, any two edges from $x$ to some $y$, $z$ in $X$ have the same color. If $X$ is infinite and defines a $1$-coloring, then, by applying $\rt^1_k$ to $X$ we produce an infinite subset $Y$ of $X$ whose points all have the same color $h$. According to the way we color points, all edges from all points of $X$ all have the color $h$. Thus, a sufficient condition for $\rt^2_k$ is the existence of an infinite set defining a $1$-coloring. In fact we need even less. Assume that $V$ is a graph whose ancestor relation is included in the complete graph $\NN$. We say that $V$ is an Erd\H{o}s' tree in $k$ colors (e.g. \cite[Definition 6.3]{Hclosure}) if for all $x \in V$, all $i=1, \dots, k$ all descendants $y$, $z$ of the child number $i$ of $x$ in $V$, the edges $x$ to $y$, $z$ have the same color number $i$. There is some Erd\H{o}s' tree recursively enumerable in the coloring (e.g. \cite{KohlKreuz1, RT22iff3LLPO}). Assume there exists some infinite $k$-ary Erd\H{o}s' tree $V$. Then $V$ has some infinite branch $r$ by K\"{o}nig's Lemma. $r$ is a total order in $V$, therefore $r$ is a complete sub-graph of $\NN$. Thus, $r$ defines an infinite $1$-coloring and proves $\rt^2_k$. Therefore a sufficient condition for $\rt^2_k$ is the existence of an infinite $k$-ary tree Erd\H{o}s' tree $V$. 

In \cite{Jockusch} Jockusch presented a modified version of Erd\H{o}s-Rado proof. Erd\H{o}s-Rado's proof, Jockusch's proof and our proof differ in the definition of $V$, although until this point they are the same. Erd\H{o}s and Rado introduce an ordering relation $\prec_E$ on $\NN$ which defines the proper ancestor relation of a $k$-ary tree $E$ on $\NN$. The $k$-coloring on edges of $\NN$, restricted to the set of pairs $x \prec_E y$, gives the same color to any two edges $x \prec_E y$ and $x \prec_E z$ with the same origin $x$. This defines an Erd\H{o}s' tree over $\NN$. In both Erd\H{o}s-Rado and Jockusch's proofs, an infinite homogeneous set is obtained from an infinite set of nodes of the same color in an infinite branch of the tree. In Erd\H{o}s-Rado and Jockusch's proofs, the Pigeonhole Principle is applied to a $\Delta^0_3$-branch obtained by K\"{o}nig's Lemma. To formalize these proofs in $\ha$ we would have to use the classical principle $\llpo{4}$: the Pigeonhole Principle for $\Delta^0_3$ predicates requires $\llpo{4}$. Our goal is to prove $\rt^2_k(\Sigma^0_0)$ using the weaker principle $\koenig$.

\begin{proposition}[$\ha + \Em_2$]\label{proposition: Erdossubtree}
For every $k \geq 2$ and for every recursive coloring $c_a: [\NN]^2 \to k$, there exists an Erd\H{o}s' tree $T$ for the coloring $c_a$ which satisfies the following properties:
\begin{itemize}
\item[{\normalfont(a)}] there exists some $\Delta^0_2$ predicate in $\ha$ which represents $x \in T$; 
\item[{\normalfont(b)}] $T$ has a unique infinite branch $r$ defined by some predicate of $\ha$;
\item[{\normalfont(c)}] if there exist infinitely many edges with color $h$ in $T$, then there are infinitely many edges with color $h$ in $r$.
\end{itemize}
\end{proposition}
\begin{proof}
Let $k\geq 2$. The standard Erd\H{o}s' tree $(\NN, \prec_E)$ associated to a coloring $c_a : [\NN]^2 \to k$ is defined as a graph, as the set of natural numbers equipped with the following relation.
\[
		x \prec_E y \iff \forall z < x (z \prec_E x \implies c_a(\bp{z,x})=c_a(\bp{z,y})).
\]
$(\NN, \prec_E)$ is recursively enumerable on the coloring, and recursive enumerable if the coloring is recursive. We would like to apply Theorem \ref{theorem:k+1} to produce an infinite branch $r$ as required, but Theorem \ref{theorem:k+1} requires a tree given as set of branches. Thus, we have to prove in $\ha$ that given a graph-tree $(N,\prec_E)$ we can extract a tree $(\tilde{E}, \prec)$ where $\tilde{E} \subset \List(k)$ which keeps all information we need. We define $(a_0, \dots, a_j) \in \tilde{E}$ if and only if there are nodes $x_0, \dots, x_{j+1} \in \NN$ such that for every $i \in j+1$ $c(x_i, x_{i+1})=a_j$ and $x_{i+1}$ is a $\prec_E$-child of $x_i$. $(\NN, \prec_E)$ contains the value of each node while the tree $(\tilde{E}, \prec)$ contains only the color of each edge, but note that given both $(\NN, \prec_E)$ and $(\tilde{E}, \prec)$, we can recursively translate any subtree of $(\tilde{E}, \prec)$ in a subtree of $(\NN, \prec_E)$. 

By applying Theorem \ref{theorem:k+1} to the $k$-ary tree $(\tilde{E}, \prec)$, the subtree $T$ of $(\NN, \prec_E)$ which corresponds to $f_C''\NN$ is $\Delta^0_2$ and has exactly one infinite branch, the rightmost. 
\end{proof}

Let $T$ be the witness of Proposition \ref{proposition: Erdossubtree}. We may prove that there are infinitely many nodes of the same color in the infinite branch of $T$ using only $\koenig$. Any infinite subset of the infinite branch of $T$ with all nodes in the same color will be some monochromatic set for the original graph. Moreover our proof recursively defines $k$-many monochromatic $\Delta^0_3$-sets, one of each color, that can not be all finite, even if we can not decide which of these is the infinite one.


\begin{theorem}\label{KoenigRamsey}
Let $k \geq 2$. Then $\koenig$ implies  $\rt^2_k(\Sigma^0_0)$ in $\ha$.
\end{theorem}

\begin{proof}

Given $T$ the witness of Proposition \ref{proposition: Erdossubtree}, we can prove Ramsey's Theorem for pairs and $k$-many colors in $\llpo{3}$.  We have to prove that the infinite branch of $T$ (which exists and it is unique by Proposition \ref{proposition: Erdossubtree}.b) has infinitely many pairs $x \prec_T y$ of color $h$. By Proposition \ref{proposition: Erdossubtree}.c, it is enough to prove that $T$ has infinitely many pairs $x \prec_T y$ of color $c$, for some $h$. By Proposition \ref{proposition: Erdossubtree}.a, $x \in T$ is a $\Delta^0_2$ predicate. Thus, if we apply the Pigeonhole Principle for $\Sigma^0_2$ predicates $(k-1)$-many times, we deduce that $T$ has infinitely many edges in color $h$ for some $h \in k$. However, the Pigeonhole Principle for $\Sigma^0_2$ predicates is a classical principle, therefore we have to derive the particular instance we use from $\koenig$.

\begin{claim}\label{principiocassetti}
$\koenig$ implies the Pigeonhole Principle for $\Sigma^0_2$.
\end{claim}

\begin{proof}[Proof Claim \ref{principiocassetti}.]
The Infinite Pigeonhole Principle for $\Sigma^0_2$ predicates can be stated as follows:
\[ \forall x \ \exists z \ [z \geq x \ \wedge \ (P(z,a) \vee Q(z,a))]\]
\[\implies \forall x \ \exists z \ [z\geq x \ \wedge \ P(z,a)] \ \vee \ \forall x \ \exists z [z \geq x \wedge Q(z,a)],\]
with $P$ and $Q$ $\Sigma^0_2$ predicates.
We prove that the formula above is equivalent in $\ha$ to some formula of $\koenig$. Let
\[
\begin{aligned}
H(x,a) &:= \exists z \ [z \geq x \wedge P(z,a)]\\
K(x,a) &:= \exists z \ [z \geq x \wedge Q(z,a)].
\end{aligned}
\]
In fact both  $H$ and $K$ are equivalent in $\ha$ to $\Sigma^0_2$ formulas $H', \ K'$. By intuitionistic prenex properties (see \cite{LICS04})
\[\exists z [z \geq x \ \wedge \ (P(z,a) \vee Q(z,a))]\]
is equivalent to
\[\exists z [z \geq x \ \wedge \ P(z,a)] \ \vee \ \exists z [z \geq x \ \wedge \ Q(z,a)].\]
The formula above is equivalent to $H' \vee K'$. Thus, any formula of Pigeonhole Principle with $P$, $Q$ $\Sigma^0_2$ is equivalent in $\ha$ to 
\[
\forall x (H'(x,a) \vee K'(x,a)) \implies \forall x H'(x,a) \vee \forall x K'(x,a),
\] 
which is the instance of $\koenig$ with $H', \ K'$. 
\end{proof}

Thus, there exist infinitely many edges of $r$ in color $h$, for some $h \in  k$. Their smaller nodes define a monochromatic set for the original graph, since given an infinite branch $r$ of an Erd\H{o}s' tree and $x \in r$, if there exists $y \in r$ such that $x \prec_T y$ and $\{x,y\}$ has color $h$, then for every $z \in r$ such that $x \prec_T z$, the edge $\{x,z\}$ has color $h$. Thus we can devise a coloring on $r$, given color $h$ to $x$ if $\{x,y\}$ has color $h$, with $y$ child of $x$ in $r$. After that, every infinite set of points with the same color  in $r$ defines an infinite set with all edges of the same color, and then it proves Ramsey's Theorem for pairs in $k$-many colors in $\ha$ starting from the assumption of $\koenig$.
\end{proof}

\section{Conclusion}

\paragraph{The complexity of the homogeneous sets.}
The infinite branch $r$ of the tree $T$ defined in Section \ref{section:3LLPOimpliesRT2n} is $\Delta^0_3$. As remarked in \cite{RT22iff3LLPO}, for some recursively enumerable tree $T$, the branch $r$ cannot be $\Delta^0_2$. Here we argue classically for short. Suppose by contradiction that $r$ is $\Delta^0_2$. In this hypothesis we will prove that for each recursive coloring there exists an infinite homogeneous set $\Delta^0_2$. Indeed, using the fact that all edges from the same point of $r$ to another point of $r$ have the same color, we may describe the homogeneous set of color $c$ as the set of points whose edges to any other point of $r$ all have color $h$:
\[\Omogeneo(y) \iff y \in r \wedge \forall z> y( \RamoInfinito(z) \implies c(\{y,z\})=h)\]
and also as the set of points having some edge to another point of $r$ of color $c$:
\[\Omogeneo(y) \iff y \in r \wedge \exists z > y( \RamoInfinito(z) \wedge c(\{y,z\})=h).\]
Therefore, if $r$ is $\Delta^0_2$ then the first formula is $\Pi^0_2$ and the second one is $\Sigma^0_2$. So for every $h$ the homogeneous set is $\Delta^0_2$. Since at least one of these sets is infinite and since Jockusch proved that exists a coloring of $[\NN]^2$ that has no infinite homogeneous set $\Sigma^0_2$ \cite{Jockusch}, we obtain a contradiction. So $r \not\in \Delta^0_2$ in general.

In Jockusch's proof he shows that one of the homogeneous sets (the red one in his notation) is $\Pi^0_2$, since at the beginning of each step he looks for red edges; while the other one are $\Delta^0_3$. In our proof we can see that all homogeneous sets are $\Delta^0_3$, since our construction is more symmetric with respect to the $k$-many colors. As a matter of fact, since $r$ is $\Delta^0_3$, the previous two formulas are respectively $\Pi^0_3$ and $\Sigma^0_3$. This is enough in order to prove that all homogeneous sets are $\Delta^0_3$. There always is an infinite homogeneous set $\Pi^0_2$, but apparently the proof is purely classical and cannot compute the integer code of such $\Pi^0_2$ predicate. Again we refer to Jockusch \cite{Jockusch} for details.

\paragraph{More about $\koenig$.}
$\koenig$ is a principle of uncommon use, but it is equivalent to K\"{o}nig's Lemma, given function variables and choice axiom \cite{LICS04}.  As shown in \cite{RT22iff3LLPO} $\llpo{n}$ is equivalent to the union of $\demorgan$ and  $\Em_{n-1}$, where
\[
	\demorgan:= \  \neg (P \wedge Q) \implies \neg P \vee \neg Q. \ (P, Q \in \Sigma^0_n)
\]

This equivalence helps us to analyse the proof of Theorem \ref{KoenigRamsey}. Indeed we can see that the most of the proof (namely Section \ref{section: T2k}) uses only $\Em_2$ and that $\demorgantre$ (and so $\koenig$) is used only to yield the Pigeonhole Principle for $\Sigma^0_2$ predicates at the end of the proof of Theorem \ref{KoenigRamsey}.

\paragraph{Further works.}
The first question that raises after this work is what is the minimal classical principle that implies $\rt^2_k(\Sigma^0_n)$, Ramsey's Theorem for pairs in two colors, but with any $\Sigma^0_n$ family of colorings. We conjecture that, modifying conveniently the proofs of $\ramsey(\Sigma^0_0) \implies \koenig$ \cite{RT22iff3LLPO} and of  $\koenig \implies \rt^2_k(\Sigma^0_0)$ (Theorem \ref{KoenigRamsey}), we should obtain that for every $k \geq 2$:
\begin{equation}\label{gen}
\llpo{(n+3)} \iff \rt^2_k(\Sigma^0_n).
\end{equation}
A first development of this paper might be to check of the equivalence \ref{gen}, for each $n \in \NN$. By increasing the size of the edges, we conjecture also that for every natural number $n\geq 2$:
\begin{equation}
\llpo{(n+1)} \iff \rt^{n}_k(\Sigma^0_0).
\end{equation}

In this paper we consider Ramsey's Theorem as schema in order to work with first order statements. Another possibility is to study Ramsey's Theorem working in $\ha + $ functions $+$ description axiom (that is a conservative extension of $\ha$, see \cite{LICS04}), in order to use only one statement to express Ramsey's Theorem for pairs in two colors. It seems to us that this unique statement is still equivalent to $\koenig$.

\subparagraph*{Acknowledgements}

The authors are thankful to Alexander Kreuzer and Paulo Oliva for their useful comments and suggestions.

\bibliographystyle{plain}
\bibliography{bibRT22iff3LLPO}

\begin{thebibliography}{10}

\bibitem{LICS04}
Yohji Akama, Stefano Berardi, Susumu Hayashi, and Ulrich Kohlenbach.
\newblock {An Arithmetical Hierarchy of the Law of Excluded Middle and Related
  Principles}.
\newblock In {\em {19th {IEEE} Symposium on Logic in Computer Science {(LICS}
  2004), 14-17 July 2004, Turku, Finland, Proceedings}}, pages 192--201. {IEEE}
  Computer Society, 2004.

\bibitem{ArgTod}
Spiros~A. Argyros and Stevo Todorcevic.
\newblock {\em {Ramsey Methods in Analysis}}.
\newblock Advanced Courses in Mathematics - CRM Barcelona. Birkh{\"a}user
  Basel, 2005.

\bibitem{RT22iff3LLPO}
Stefano Berardi and Silvia Steila.
\newblock {Ramsey Theorem for Pairs As a Classical Principle in Intuitionistic
  Arithmetic}.
\newblock In Ralph Matthes and Aleksy Schubert, editors, {\em {19th
  International Conference on Types for Proofs and Programs, {TYPES} 2013,
  April 22-26, 2013, Toulouse, France}}, volume~26 of {\em LIPIcs}, pages
  64--83. Schloss Dagstuhl - Leibniz-Zentrum fuer Informatik, 2014.

\bibitem{Hclosure}
Stefano Berardi and Silvia Steila.
\newblock {An intuitionistic version of Ramsey's Theorem and its use in Program
  Termination}.
\newblock {\em Annals of Pure and Applied Logic}, 166(12):1382--1406, 2015.

\bibitem{Bishop}
Errett Bishop.
\newblock {\em Foundations of constructive analysis}.
\newblock McGraw-Hill Book Co., 1967.

\bibitem{Buchi}
Julius~Richard B{\"u}chi.
\newblock On a decision method in restricted second order arithmetic.
\newblock In {\em Logic, Methodology and Philosophy of Science (Proc. 1960
  Internat. Congr .)}, pages 1--11. Stanford Univ. Press, Stanford, Calif.,
  1962.

\bibitem{Coquand}
Thierry Coquand.
\newblock {A direct proof of Ramsey's Theorem}.
\newblock Author's website, revised in 2011, 1994.

\bibitem{Gasarch}
William~I. Gasarch.
\newblock {Proving Programs Terminate Using Well-Founded Orderings, Ramsey's
  Theorem, and Matrices}.
\newblock {\em Advances in Computers}, 97:147--200, 2015.

\bibitem{Jockusch}
Carl~G. Jockusch, Jr.
\newblock Ramsey's theorem and recursion theory.
\newblock {\em Journal of Symbolic Logic}, 37:268--280, 1972.

\bibitem{KohlKreuz1}
Alexander Kreuzer and Ulrich Kohlenbach.
\newblock {Ramsey's Theorem for Pairs and Provably Recursive Functions}.
\newblock {\em Notre Dame Journal of Formal Logic}, 50(4):427--444, 2009.

\bibitem{Jones1}
Chin~Soon Lee, Neil~D. Jones, and Amir~M. Ben{-}Amram.
\newblock The size-change principle for program termination.
\newblock In Chris Hankin and Dave Schmidt, editors, {\em {Conference Record of
  {POPL} 2001: The 28th {ACM} {SIGPLAN-SIGACT} Symposium on Principles of
  Programming Languages, London, UK, January 17-19, 2001}}, pages 81--92.
  {ACM}, 2001.

\bibitem{Podelski}
Andreas Podelski and Andrey Rybalchenko.
\newblock {Transition Invariants}.
\newblock In {\em {19th {IEEE} Symposium on Logic in Computer Science {(LICS}
  2004), 14-17 July 2004, Turku, Finland, Proceedings}}, pages 32--41. {IEEE}
  Computer Society, 2004.

\bibitem{Ramsey}
Frank~P. Ramsey.
\newblock On a problem in formal logic.
\newblock {\em Proceedings of the London Mathematical Society}, 30:264--286,
  1930.

\bibitem{Specker}
Ernst Specker.
\newblock Ramsey's theorem does not hold in recursive set theory.
\newblock In {\em Logic {C}olloquium '69 ({P}roc. {S}ummer {S}chool and
  {C}olloq., {M}anchester, 1969)}, pages 439--442. North-Holland, Amsterdam,
  1971.

\bibitem{SCTTypes}
Silvia Steila.
\newblock {An Intuitionistic Analysis of Size-change Termination}.
\newblock In Hugo Herbelin, Pierre Letouzey, and Matthieu Sozeau, editors, {\em
  {20th International Conference on Types for Proofs and Programs, {TYPES}
  2014, May 12-15, 2014, Paris, France}}, volume~39 of {\em LIPIcs}, pages
  288--307. Schloss Dagstuhl - Leibniz-Zentrum fuer Informatik, 2015.

\bibitem{CoquandStop}
Dimitrios Vytiniotis, Thierry Coquand, and David Wahlstedt.
\newblock {Stop When You Are Almost-Full - Adventures in Constructive
  Termination}.
\newblock In Lennart Beringer and Amy~P. Felty, editors, {\em {Interactive
  Theorem Proving - Third International Conference, {ITP} 2012, Princeton, NJ,
  USA, August 13-15, 2012. Proceedings}}, volume 7406 of {\em Lecture Notes in
  Computer Science}, pages 250--265. Springer, 2012.

\end{thebibliography}
\end{document}